\RequirePackage{fix-cm}
\documentclass[smallextended]{svjour3}       % onecolumn (second format)
\smartqed
\usepackage{graphicx}
\usepackage{mathptmx}

\journalname{Numerical Algorithms}

\usepackage{amssymb}
\usepackage{amsmath}
\usepackage{algorithm}
\usepackage{algpseudocode}
\makeatletter
\newcommand{\algmargin}{\the\ALG@thistlm}
\makeatother
\algnewcommand{\ParState}[1]{\State\parbox[t]{\dimexpr\linewidth-\algmargin}{\strut #1\strut}}
\usepackage{booktabs}
\usepackage{lscape}

\usepackage{mathtools}
\mathtoolsset{showonlyrefs = true}

\usepackage{xcolor}

\DeclareMathOperator*{\argmin}{argmin}
\DeclareMathOperator*{\argmax}{argmax}

\DeclareMathOperator*{\interior}{int}

\newcommand{\calK}{\mathcal{K}}
\newcommand{\bbR}{\mathbb{R}}

\spnewtheorem{assumption}{Assumption}{\bf}{\rm}

\begin{document}
\title{Extension of the LP-Newton method to SOCPs via semi-infinite representation%
	\thanks{This research is supported by JST CREST JPMJCR14D2 and JSPS Grants-in-Aid for Young Scientists 15K15943 and 16K16357. We thank Stuart Jenkinson, PhD, from Edanz Group (www.edanzediting.com/ac) for editing a draft of this manuscript. }
}
\titlerunning{Extension of the LP-Newton method to SOCP}
\author{Takayuki~Okuno \and Mirai~Tanaka}
\institute{T.~Okuno \at
	RIKEN Center for Advanced Intelligence Project,
	1-4-1 Nihonbashi, Chuo, Tokyo 103-0027, JAPAN.\\
	\email{takayuki.okuno.ks@riken.jp}
	\and
	M.~Tanaka \at
	Department of Statistical Inference and Mathematics,
	The Institute of Statistical Mathematics,
	10-3, Midori, Tachikawa, Tokyo 190-8562, JAPAN.
}
\date{Received: date / Accepted: date}

\maketitle

\begin{abstract}
The LP-Newton method solves the linear programming problem (LP)
by repeatedly projecting a current point onto a certain relevant polytope.
In this paper, we extend the algorithmic framework of the LP-Newton method 
to the second-order cone programming problem (SOCP)
via a linear semi-infinite programming (LSIP) reformulation of the given SOCP.
In the extension, we produce a sequence by projection onto polyhedral cones
constructed from LPs obtained by finitely relaxing the LSIP.
We show the global convergence property of the proposed algorithm under mild assumptions,
and investigate its efficiency through numerical experiments
comparing the proposed approach with the primal-dual interior-point method for the SOCP.
\end{abstract}

\section{Introduction}
\label{sect: intro}

In this paper, we consider the following second-order cone programming problem (SOCP):
\begin{equation}
\label{eq:SOCPstd}
\begin{array}{ll}
\text{maximize}&	c^{\top} x\\
\text{subject to}&	A x = b,\\
&					x \in \calK,
\end{array}
\end{equation}
where $A \in \bbR^{m \times n}, b \in \bbR^{m}$, and $c \in \bbR^{n}$
are a given matrix and vectors, and
$\calK$ denotes a Cartesian product of second-order cones (SOCs), i.e.,
$\calK = \calK^{n_{1}} \times \calK^{n_{2}} \times \dots \times \calK^{n_{p}}$
with $\calK^{l}$ being an $l$-dimensional SOC, namely,
\begin{equation}
\calK^{l} :=
\begin{cases}
\left\{(z_{1}, z_{2}, \dots, z_{l}) \in \bbR^{l}: z_{1} \ge \sqrt{\displaystyle \sum_{j = 2}^{l} z_{j}^{2}}\right\}
	&(l \ge 2),\\
\{z \in \bbR: z \ge 0\}
	&(l = 1).
\end{cases}
\end{equation}
%In the case where
If the SOCs in $\calK$ are all one-dimensional,
then the SOCP problem~\eqref{eq:SOCPstd} reduces to the linear programming problem (LP) of the standard form:
\begin{equation}
\label{eq:LPstd}
\begin{array}{ll}
\text{maximize}&	c^{\top} x\\
\text{subject to}&	A x = b,\\
&					x \ge 0.
\end{array}
\end{equation}
SOCP~\eqref{eq:SOCPstd} is a very important optimization model, as it has many practical applications
in fields such as robust optimization, antenna array problems, and beam forming problems~\cite{lobo1998applications}.
%From the viewpoint of the Jordan algebra\cite{fau},
%the SOC can be classified as a symmetric cone and the SOCP can be regarded as the middle class between the linear program and semidefinite program.
To solve the SOCP, many researchers have developed algorithms exploiting the
geometrical or algebraic structure of SOCs.
For instance, we can find Newton-type methods such as primal-dual interior-point methods~\cite{monteiro2000polynomial}
and non-interior continuous methods along with complementarity functions~\cite{hayashi1},
Chubanov-type algorithms~\cite{kitahara2018extension},
and simplex-type algorithms~\cite{hayashi2016simplex,muramatsu2006pivoting}.
These algorithms were originally carried over from LP.

One popular extension from LP to SOCP is based on the Jordan algebra~\cite{faraut1994analysis},
whereby the two problems can be handled in the same algebraic framework.
Another approach is based on the semi-infinite reformulation of the SOCP.
By representing the SOCs as the intersection of an infinite number of half-spaces,
the SOCP can be reformulated as the following linear semi-infinite programming problem (LSIP)
with infinitely many linear inequality constraints:
\begin{equation}
\label{eq:SIP}
\begin{array}{lll}
\text{maximize}&	c^{\top} x\\
\text{subject to}&	A x = b,\\
&					(1, (v^{i})^{\top}) x^{i} \ge 0	&(v^{i} \in V_{i};\ i = 1, 2, \dots, p),
\end{array}
\end{equation}
where $V_{i} := \{v \in \bbR^{n_{i} - 1} \mid \|v\| \le 1\}$ if $n_{i} \ge 2$;
otherwise, the corresponding constraint denotes $x^{i} \ge 0$ by convention,
%if $n_{i} = 1$ for convention
and $x^{i} \in \bbR^{n_{i}}$ denotes the $i$-th block of $x$ partitioned along the Cartesian structure of $\calK$, i.e.,
$x = ((x^{1})^{\top}, (x^{2})^{\top}, \dots, (x^{p})^{\top})^{\top} \in \prod_{i = 1}^{p} \bbR^{n_{i}}$.
Hayashi et al.~\cite{hayashi2016simplex} tailored the dual-simplex method for LP to the dual problem of SOCP~\eqref{eq:SOCPstd}
%via the semi-infinite representation~\eqref{eq:SIP}.
via the semi-infinite representation.
For an overview of semi-infinite programming problems, we refer readers to survey articles~\cite{sip1,sip2}.

The purpose of this paper is to extend the LP-Newton method for LP in the standard form~\eqref{eq:LPstd} to SOCP~\eqref{eq:SOCPstd}.
Algorithms for solving LP include the simplex method, ellipsoid method, and interior-point method.
Although the ellipsoid and interior-point methods are polynomial-time algorithms,
the existence of a strongly polynomial-time algorithm for solving LPs remains an open problem.
In an attempt to devise a strongly polynomial-time algorithm for LPs, Fujishige et al.~\cite{FHYZ09} proposed the LP-Newton method for box-constrained LPs,
which have a box constraint~$l \le x \le u$ instead of the nonnegativity constraint in LP~\eqref{eq:LPstd}.
Kitahara et al.~\cite{KMS13} extended this to LPs in the standard form~\eqref{eq:LPstd}.
This algorithm repeats the projection of the current point onto a polytope arising from the feasible region
and the computation of a supporting hyperplane and line.
Numerical results in \cite{FHYZ09} suggest that relatively few iterations of the LP-Newton method are required, 
and hence the algorithm is considered promising.
%The LP-Newton method was extended to standard form LP~\eqref{eq:LPstd} by Kitahara et al.~\cite{KMS13}

Recently, Silvestri and Reinelt~\cite{SR17} developed an LP-Newton method for SOCP.
To the best of our knowledge, this is the first extension of the LP-Newton method to SOCP.
In \cite{SR17}, the authors considered SOCP~\eqref{eq:SOCPstd}
with $x \in \calK$ replaced by a box-like constraint $l \preceq x \preceq u$,
which denotes $x - l, u - x \in \calK$. 
%instead of $x \in \calK$ in SOCP~\eqref{eq:SOCPstd}.
Their algorithm computes a projection onto a conic zonotope at each iteration,
and they proposed a Frank--Wolfe-based inner algorithm for this computation.
%However,
%the projection onto a conic zonotope seems difficult.
Nevertheless, the computation of the projection still appears to be difficult.
In fact, their numerical results show that the inner algorithm for obtaining the projection requires a number of iterations,
although the outer loop is repeated relatively few times.

In this paper, we propose a different type of LP-Newton method for SOCP~\eqref{eq:SOCPstd}
based on the semi-infinite reformulation~\eqref{eq:SIP}.
In our approach, we construct a sequence of LPs by
adaptively selecting finitely many constraints from the infinitely many constraints of LSIP~\eqref{eq:SIP}.
%\to which we apply the LP-Newton method for an LP successively.
To produce an iteration point,
we compute a projection onto a polytope arising from a polyhedral approximation of the SOCs,
which can be realized by solving a convex quadratic programming problem (QP).
%As a result, the computation of the projection at each iteration can be realized by solving a quadratic program (QP).
%To this end, we once reformulate SOCP~\eqref{eq:SOCPstd}
%as an equivalent semi-infinite program of the following form:
%\begin{equation}
%\label{eq:SIP}
%\begin{array}{lll}
%\text{maximize}&	c^{\top} x\\
%\text{subject to}&	A x = b,\\
%&					(1, (v^{i})^{\top}) x^{i} \ge 0	&(\forall v^{i} \in V_{i};\ i = 1, 2, \dots, p),
%\end{array}
%\end{equation}
%where $V_{i} = \{v \in \bbR^{n_{i} - 1} \mid \|v\| \le 1\}$ if $n_{i} \ge 2$,
%otherwise $V_{i} = \emptyset$;
%%if $n_{i} = 1$ for convention
%and $x^{i} \in \bbR^{n_{i}}$ denotes the $i$-th block of $x$ partitioned along $\calK$, i.e.,
%$x = ((x^{1})^{\top}, (x^{2})^{\top}, \dots, (x^{p})^{\top})^{\top} \in \prod_{i = 1}^{p} \bbR^{n_{i}}$.
%$x$ is expressed as \bbR^{n_{1}} \times \bbR^{n_{2}} \times \dots \bbR^{n_{p}}$.

The remainder of this paper is organized as follows. 
In Section~\ref{sect: primal}, we describe our proposed LP-Newton method for SOCP~\eqref{eq:SOCPstd}.
In Section~\ref{sect: conv anal}, we establish the global convergence of the proposed algorithm
under the boundedness of the optimal set of SOCP~\eqref{eq:SOCPstd}.
In Section~\ref{sect: dual}, we propose a dual algorithm that generates a sequence in the dual space of SOCP~\eqref{eq:SOCPstd}.
We also show its global convergence to an optimum of the dual problem of SOCP~\eqref{eq:SOCPstd}
under Slater's constraint qualification.
In Section~\ref{sect: expr}, we report numerical results for the proposed method
to investigate its validity and effectiveness.

\section{Primal algorithm}
\label{sect: primal}

In this section,
we extend the LP-Newton method for LP~\eqref{eq:LPstd} proposed by Kitahara et al.~\cite{KMS13} to SOCP~\eqref{eq:SOCPstd}.
For simplicity, we use the following notation:
\begin{equation}
\bar{A} :=
	\begin{pmatrix}
	c^{\top}\\
	A
	\end{pmatrix}
	\in \bbR^{(1 + m) \times n},\
L := \left\{
	\begin{pmatrix}
	\gamma\\
	b
	\end{pmatrix}:
	\gamma \in \bbR\right\},
%\bar{v}^{i} :=
%	\begin{pmatrix}
%	1\\
%	v^{i}
%	\end{pmatrix},\
%\bar{v}^{\top}=[\bar{v}_1^{\top},\dots,\bar{v}_p^{\top}],\\
%K :=
%	\{\bar{A} x: x\in \calK\}
\end{equation}
and for some $E_{i} \subseteq V_{i}\ (i = 1, 2, \dots, p)$,
\begin{equation}
E :=
	\prod_{i = 1}^{p} E_{i},
%	\left(\subseteq \prod_{i = 1}^{p} V_{i}\right),\\
\calK_{E} :=
	\{x \in \bbR^{n}: (1, (v^{i})^{\top}) x^{i} \ge 0\
	(\forall v^{i} \in E_{i}; i = 1, 2, \dots, p)\}.
%K_{E} :=
%	\{\bar{A} x: x \in \calK_{E}\}.
\end{equation}
Moreover, we often denote $(1, v^{\top})^{\top}$ by $(1; v)$ for any vector~$v$.

%Our proposed algorithm is named \textit{adaptive LP-Newton (ALPN) method} for SOCP~\eqref{eq:SOCPstd}
%after changing the form of LP adaptively.
%The adaptive LP-Newton method is formally described as Algorithm~\ref{alg:ALPN}.
In the proposed algorithm, we construct a sequence of outer polyhedral approximations of the SOCs.
By applying the LP-Newton method to the resulting LP,
we update the polyhedral approximation of the SOCs.
As a result, the algorithm generates a sequence~$\{\calK_{E^{(k)}}\}$ of adaptive outer approximations of the SOCs
and a sequence~$\{x^{(k)}\}$ of approximate optimal solutions to SOCP~\eqref{eq:SOCPstd}.
We name the proposed algorithm the \textit{adaptive LP-Newton (ALPN) method} for SOCP~\eqref{eq:SOCPstd}
%after changing form of LPs adaptively.
%The adaptive LP-Newton method is formally described as Algorithm~\ref{alg:ALPN}.
and formally describe it as Algorithm~\ref{alg:ALPN}.

\begin{algorithm}
\caption{Adaptive LP-Newton method for SOCP~\eqref{eq:SOCPstd}}
\label{alg:ALPN}
\begin{algorithmic}[1]
\State{Choose initial finite sets~$E_{i}^{(0)} \subseteq V_{i}$ for $i = 1, 2, \dots, p$ such that
$\argmax\{c^{\top} x: A x = b, x \in \calK_{E^{(0)}}\} \ne \emptyset$.}
%$\mathrm{LP}(E^{(0)})$ with $E^{(0)} := \prod_{i = 1}^{p} E_{i}^{(0)}$ has a nonempty optimal solution set.
\State{}
\Comment{If no such $E^{(0)}$ exists, SOCP~\eqref{eq:SOCPstd} is infeasible or unbounded.}
\State{Choose an initial point~$w^{(0)} := (\gamma^{(0)}, b) \in L$ with sufficiently large $\gamma^{(0)} \in \bbR$.}
\State{Set $k := 0$.}
\Loop
	\ParState{%
%	\State{%
	Find the nearest point $\bar{w}^{(k)} := (\zeta^{(k)}, b^{(k)}) \in \bbR \times \bbR^{m}$ of $w^{(k)}$
	in $\bar{A} \calK_{E^{(k)}} = \{\bar{A} x: x \in \calK_{E^{(k)}}\}$
	and obtain $x^{(k)}$ such that $\bar{A} x^{(k)} = \bar{w}^{(k)}$ and
	$x^{(k)} \in \calK_{E^{(k)}}$.}
	\If{$x^{(k)}$ is feasible for SOCP~\eqref{eq:SOCPstd}}
		\State{\Return{} $x^{(k)}$.}
		\Comment{$x^{(k)}$ is an optimal solution of SOCP~\eqref{eq:SOCPstd}.}
	\EndIf
	\ParState{%
%	\State{%
	Let $w^{(k + 1)} := (\gamma^{(k + 1)}, b)$ be the intersection point of $L$ and $H^{(k)}$,
	where $H^{(k)}$ is the supporting hyperplane of $\bar{A} \calK_{E^{(k)}}$ on $\bar{w}^{(k)}$
	orthogonal to $w^{(k)} - \bar{w}^{(k)}$.}
	\For{$i = 1, 2, \dots, p$}
	\ParState{%
%	\State{%
		Find $v^{i, (k)} \in V_{i}$ such that
%		violating the SOC constraint \tred{$x^{i} \in \calK^{n_{i}}$} most at $x^{i, (k)}$, namely,
		\begin{equation}
		v^{i, (k)} \in \argmin_{v^{i} \in V_{i}} (1, (v^{i})^{\top}) x^{i, (k)}
		\label{eq:cut}
		\end{equation}
		and set $E_{i}^{(k + 1)} := E_{i}^{(k)} \cup \{v^{i, (k)}\}$.}
		\State{}
		\Comment{If $(1, (v^{i, (k)})^{\top}) x^{i, (k)} < 0$, $v^{i, (k)}$ violates $x^{i} \in \calK^{n_{i}}$ most at $x^{i, (k)}$.}
	\EndFor
	\State{$k := k + 1$.}
\EndLoop
\end{algorithmic}
\end{algorithm}

In the computation of~$\bar{w}^{(k)}$ and $x^{(k)}$ in Algorithm~\ref{alg:ALPN},
%we may solve the following convex quadratic program:
%\begin{equation}
%\begin{array}{lll}
%\text{minimize}&	\|w^{(k)} - w\|^{2}\\
%\text{subject to}&	\bar{A} x = w,\\
%&					(1, (v^{i})^{\top}) x^{i} \ge 0	&(v^{i} \in E_{i}^{(k)}; i = 1, 2, \dots, p).
%\end{array}
%\label{eq: projection}
%\end{equation}
we may solve the following QP:
\begin{equation}
\begin{array}{lll}
\text{minimize}&	\|\bar{A} x - w^{(k)}\|^{2}\\
\text{subject to}&	(1, (v^{i})^{\top}) x^{i} \ge 0	&(\forall v^{i} \in E_{i}^{(k)}; i = 1, 2, \dots, p),
\end{array}
\label{eq: projection}
\end{equation}
use an optimal solution as $x^{(k)}$, and set $\bar{w}^{(k)} = \bar{A} x^{(k)}$.
If we solve QP~\eqref{eq: projection} using the active set method,
we can set $(\bar{w}^{(k)}, x^{(k)})$ as an initial point of the $(k - 1)$-th iteration.
Despite the existence of a warm-start technique,
solving QPs is still computationally expensive.
Hence, a more sophisticated subroutine may be required.
The LP-Newton method~\cite{FHYZ09} for a box-constrained LP employs Wolfe's algorithm~\cite{Wol76}
to find the nearest point in a zonotope to a given point,
and the LP-Newton method~\cite{KMS13} for the standard form LP~\eqref{eq:LPstd} uses Wilhelmsen's algorithm~\cite{Wil76}
to find the nearest point in a polyhedral cone to a given point.
The subroutines are conjectured to be polynomial-time algorithms,
and thus the LP-Newton methods for LPs have the potential to be strongly polynomial-time algorithms.
Although these subroutines are powerful,
it may be difficult to use them in Algorithm~\ref{alg:ALPN}.
In these subroutines, extreme directions or points of the zonotope or polyhedral cone are explicitly required.
In our case, unfortunately, we do not have such explicit formulas.

Note that we can compute $\gamma^{(k + 1)}$ in Algorithm~\ref{alg:ALPN} by
\begin{equation}
%\gamma^{k + 1} = \zeta^{(k)} - \frac{\|b - \bar{w}^{(k)}\|^{2}}{\gamma^{(k)} - \zeta^{(k)}},
\gamma^{(k + 1)} = \zeta^{(k)} - \frac{\|b - b^{(k)}\|^{2}}{\gamma^{(k)} - \zeta^{(k)}}.
\end{equation}
%where $\zeta^{(k)}$ is the first component of $\bar{w}^{(k)}$.
In addition, it is easy to compute $v^{i, (k)}$ in Algorithm~\ref{alg:ALPN}.
In fact, an optimal solution to Problem~\eqref{eq:cut} can be written in the following closed form:
\begin{equation}
v^{i, (k)} = -\frac{\bar{x}^{i, (k)}}{\|\bar{x}^{i, (k)}\|}
\end{equation}
if $\bar{x}^{i, (k)} \neq 0$,
%where $\bar{x}_{n_{i}}^{(k)} \in \bbR^{n_{i} - 1}$ denotes the subvector of $x_{x_{i}}^{(k)}$ without the first element,
%that is, $\bar{x}_{n_{i}}^{(k)} = ([x^{i}^{(k)}]_{2}, [x^{i}^{(k)}]_{3}, \dots, [x^{i}^{(k)}]_{m_{i}})$.
where $\bar{x}^{i, (k)} \in \bbR^{n_{i} - 1}$ denotes the subvector of $x^{i, (k)}$ without the first element,
that is, $x^{i, (k)} = (x_{1}^{i, (k)}, \bar{x}^{i, (k)}) \in \bbR \times \bbR^{n_{i} - 1}$.

\section{Convergence analysis}
\label{sect: conv anal}

%\subsection{The case of the SOCP with a nonempty optimal solution set}
In this section, we prove that
a generated sequence converges globally to an optimum of SOCP~\eqref{eq:SOCPstd}.
To this end, we make the following assumption:

\begin{assumption}
\label{H1}
The optimal solution set~$\mathcal{S}_{\mathrm{opt}}^{\mathrm{P}}$ of SOCP~\eqref{eq:SOCPstd} is nonempty and compact.
\end{assumption}

\begin{remark}
Assumption~\ref{H1} holds if the dual problem of \eqref{eq:SOCPstd}:
%$\min -b^{\top} y \text{ s.t. } c - A^{\top} y \in \calK$
\begin{equation}
\begin{array}{ll}
\text{minimize}&	b^{\top} y\\
\text{subject to}&	A^{\top} y - c \in \calK.
\end{array}
\label{eq:SOCPdual}
\end{equation}
has an optimum and strictly feasible solution, i.e.,
there exists some $\bar{y} \in \bbR^{m}$ such that $A^{\top} \bar{y} -c\in \interior \calK$.
\end{remark}

We first state the following technical lemmas.
%The proof is quite similar to that of \cite[Lemma~3.1.]{KMS13}, so we omit it.

\begin{lemma}
\label{lem:zinK}
Let $\{s^{(k)}\}$ be a sequence of nonnegative scalars 
and $\{z^{(k)}\}$ be defined by $z^{(k)} := s^{(k)}x^{(k)}$ for each $k \ge 0$.
If there exists an accumulation point of $\{z^{(k)}\}$, then it belongs to $\calK$.
\end{lemma}

\begin{proof}
Let us express $z^{(k)}$ as
%\begin{align}
%z^{(k)}
%&:= ((z^{1})^{(k)}, (z^{2})^{(k)}, \dots, (z^{p})^{(k)})\\
%&= s^{(k)} ((x^{1})^{(k)}, (x^{2})^{(k)}, \dots, (x^{p})^{(k)}).
%\end{align}
$z^{(k)}
:= (z^{1, (k)}, z^{2, (k)}, \dots,z^{p, (k)})
= s^{(k)} (x^{1, (k)}, x^{2, (k)}, \dots, x^{p, (k)})$
and let $z^{\ast} := (z^{1, \ast}, z^{2, \ast}, \dots,z^{p, \ast})$ be an arbitrarily chosen accumulation point of $\{z^{(k)}\}$.
As $V^{i}$ is compact for every $i = 1, 2, \dots, p$,
$\{v^{(k)}\} \subseteq V = \prod_{i = 1}^{p} V_{i}$ has at least one accumulation point in $V$.
Denote this point by $v^{\ast}$ and express it as $(v^{1,\ast}, v^{2, \ast}, \dots, v^{p, \ast}) \in \prod_{i = 1}^{p} V_{i}$.
Taking an appropriate subsequence $\{v^{(k)}\}_{k \in S}$,
%without loss of generality,
%$S$,
we can assume that $(v^{(k)}, z^{(k)})$ converges to $(v^{\ast}, z^{\ast})$ as $k$ tends to $\infty$ in $S$.
%$$\lim_{k \in S \to \infty} (v^{(k)}, z^{(k)}) = (v^{\ast}, z^{\ast})$$ without loss of generality.

Note that $v^{i, (k)} \in \argmin_{v^{i} \in V_{i}} (1, (v^{i})^{\top}) z^{i, (k)}$ holds
because $v^{i, (k)} \in \argmin_{v^{i} \in V_{i}} (1, (v^{i})^{\top}) x^{i, (k)}$ and $s^{(k)} \ge 0$.
Here, by letting $k\to \infty$ in $S$,
we obtain $v^{i, \ast} \in \argmin_{v^{i} \in V_{i}} (1, (v^{i})^{\top}) z^{i, \ast}$.
Thus, to show $z^{\ast} \in \calK$, it suffices to prove that
$(1, (v^{i, \ast})^{\top}) z^{i, \ast} \ge 0$ for each $i$.
%$((v^{i})^{\ast})^{\top} (z^{i})^{\ast} \ge 0$ for each $i$.
%Notice that $\langle x^{(k)}},v^{(k)}\rangle \ge 0$
To this end, let us fix $i$ and prove $ (1, (v^{i, (k)})^{\top}) z^{i, \ast} \ge 0$ for any $k \in S$.
Choosing some arbitrary $\hat{k} \in S$, it follows that $(1, (v^{i, (\hat{k})})^{\top}) z^{i, (k)} \ge 0$ for any $k > \hat{k}$ in $S$,
because $z^{i, (k)} = s^{(k)} x^{i, (k)} \in \calK_{E^{(k)}}$
and $v^{i, (\hat{k})} \in E_{i}^{(k)}$ for any $k > \hat{k}$.
Then, by letting $k$ tend to $\infty$ in $S$,
we obtain $(1, (v^{i, (\hat{k})})^{\top}) z^{i, \ast} \ge 0$.
As $\hat{k}$ was arbitrarily chosen from $S$,
%we may replace it with $k \in S$,
we conclude that $(1, (v^{i, (k)})^{\top}) z^{i, \ast} \ge 0\ (k \in S)$ holds.
Finally, by forcing $k \in S \to \infty$,
we conclude that $(1, (v^{i, \ast})^{\top}) z^{i, \ast} \ge 0$ for any $i$.
Therefore, $z^{\ast} \in \calK$.
\qed
\end{proof}

\begin{lemma}
\label{lem:obj}
Let $\theta^{\ast}$ be the optimal value of SOCP~\eqref{eq:SOCPstd}.
If the algorithm does not stop at the $k$-th iteration, we have
\begin{equation}
\theta^{\ast} \le \gamma^{(k + 1)} \le \zeta^{(k)} \le \gamma^{(k)}.
\end{equation}
\end{lemma}

\begin{proof}
Let $\theta^{(k)} \in \bbR$ be the optimal value of LSIP~\eqref{eq:SIP}
with $V_{i}$ replaced by $E_{i}^{(k)}$ for $i = 1, 2, \dots, p$,
which is a relaxation problem for SOCP~\eqref{eq:SOCPstd}.
Therefore, $\theta^{\ast} \le \theta^{(k)}$ holds.
In a similar manner to \cite[Lemma~3.1]{KMS13}, it can be verified that
\begin{equation}
\theta^{(k)}\le \gamma^{(k + 1)} \le \zeta^{(k)} \le \gamma^{(k)}.
\end{equation}
Hence, we have the desired result.
\qed
\end{proof}

\begin{lemma}
\label{lem:limb}
$\lim_{k \to \infty} b^{(k)} = b$ holds.
\end{lemma}

\begin{proof}
To show the desired result, we prove $\lim_{k \to \infty} \|b^{(k)} - b\| = 0$.
Note that $\{\gamma^{(k)}\}$ and $\{\zeta^{(k)}\}$ converge to the same point by Lemma~\ref{lem:obj}.
Thus, $|\gamma^{(k)}-\gamma^{(k + 1)}| \to 0$ and $|\gamma^{(k + 1)} - \zeta^{(k)}| \to 0$ as $k \to \infty$.
Moreover, as $\bar{w}^{(k)}$ is the projection of $w^{(k)}$ onto the supporting hyperplane $H^{(k)}$ and ${w}^{(k + 1)}\in H^{(k)}$,
we have 
\begin{equation}
|\gamma^{(k)} - \gamma^{(k + 1)}|
= \|w^{(k)} - w^{(k + 1)}\|
\ge \|w^{(k)} - \bar{w}^{(k)}\|
= \left\|
\begin{pmatrix}
\gamma^{(k + 1)} - \zeta^{(k)}\\
b^{(k)} - b
\end{pmatrix}
\right\|.
\end{equation}
From these facts, it follows that $\lim_{k \to \infty} \|b^{(k)} - b\| = 0$.
\qed
\end{proof}

\begin{proposition}
\label{prop:boundx}
If Assumption~\ref{H1} holds,
then the generated sequence $\{x^{(k)}\}$ is bounded.
\end{proposition}

\begin{proof}
Denote the feasible domain of SOCP~\eqref{eq:SOCPstd} by $\mathcal{F}$.
%%Note that from Assumption~\ref{H1},
%%the level set $\Omega:=\mathcal{F} \cap \{x\mid c^{\top}x\le \gamma_0\}$ is nonempty and bounded.
To show the boundedness of $\{x^{(k)}\}$, we assume to the contrary for a contradiction.
Thus, there exists some subsequence $\{x^{(k)}\}_{k \in S}\subseteq \{x^{(k)}\}$
such that $\lim_{k \in S \to \infty} \|x^{(k)}\| = \infty$ and $\|x^{(k)}\| \neq 0$ for any $k \in S$.
Then, we have
\begin{equation}
\frac{A x^{(k)}}{\|x^{(k)}\|} = \frac{b^{(k)}}{\|x^{(k)}\|},\
\frac{x^{(k)}}{\|x^{(k)}\|} \in \calK_{E^{(k)}},\
\frac{c^{\top} x^{(k)}}{\|x^{(k)}\|} \le \frac{\gamma^{(0)}}{\|x^{(k)}\|},
\end{equation}
where the second relation is derived from the fact that $x^{(k)} \in \calK_{E^{(k)}}$ and $\calK_{E^{(k)}}$ is a cone.
Letting $k$ tend to $\infty$ in the above and choosing an arbitrary accumulation point of $\{x^{(k)} / \|x^{(k)}\|\}$, denoted by $d^{\ast}$, implies that
\begin{equation}
A d^{\ast} = 0,\
d^{\ast} \in \calK,\
c^{\top} d^{\ast} \le 0,\
\|d^{\ast}\| = 1,
\label{eq1}
\end{equation}
where the first relation follows from the boundedness of $\{b^{(k)}\}$ implied by Lemma~\ref{lem:limb}
and the second one follows from Lemma~\ref{lem:zinK} with $z^{(k)} = x^{(k)} / \|x^{(k)}\|$ and $s^{(k)} = 1 / \|x^{(k)}\|$.
Choose $\bar{z} \in \mathcal{S}_{\mathrm{opt}}^{\mathrm{P}}$ arbitrarily and define $\Omega := \{\bar{z} + s d^{\ast}: s \ge 0\}$.
We then deduce that $\Omega \subseteq \mathcal{S}_{\mathrm{opt}}^{\mathrm{P}}$ from Equation~\eqref{eq1},
because
\begin{equation}
A (\bar{z} + s d^{\ast}) = b,\
\bar{z} + s d^{\ast} \in \calK,\
c^{\top} (\bar{z} + s d^{\ast}) \le c^{\top} \bar{z}
\end{equation}
for any $s \ge 0$, where the second statement follows from the facts that
$\bar{z} \in \calK$, $s d^{\ast} \in \calK$, and $\calK$ is a convex cone.
Note that $\Omega$ is unbounded because $\|d^{\ast}\| = 1$,
which implies the unboundedness of $\mathcal{S}_{\mathrm{opt}}^{\mathrm{P}}$.
However, this contradicts Assumption~\ref{H1}.
As a consequence, $\{x^{(k)}\}$ is bounded.
\qed
\end{proof}

\begin{theorem}
\label{thm:xopt}
If Assumption~\ref{H1} holds,
any accumulation point of $\{x^{(k)}\}$ is an optimum of SOCP~\eqref{eq:SOCPstd}.
\end{theorem}

\begin{proof}
From Proposition~\ref{prop:boundx}, $\{x^{(k)}\}$ is bounded and has an accumulation point.
Choose an arbitrary accumulation point and denote it by $x^{\ast}$.
Without loss of generality, we may assume that $\lim_{k \to \infty} x^{(k)} = x^{\ast}$.
Now, let us recall that
%\begin{equation}
$A x^{(k)} = b^{(k)}$ and $x^{(k)} \in \calK_{E^{(k)}}$
%\end{equation}
hold for any $k$. Together with Lemmas~\ref{lem:zinK} and \ref{lem:limb}, this implies
that $A x^{\ast} = b$ and $x^{\ast} \in \calK$, that is, $x^{\ast}$ is feasible for the SOCP.
Hence, $c^{\top} x^{\ast} \le \theta^{\ast}$ follows, where $\theta^{\ast}$ denotes the optimal value of the SOCP.
However, by Lemma~\ref{lem:obj},
it holds that $\theta^{\ast} \le \zeta^{(k)} = c^{\top} x^{(k)}$ for any $k$,
and by taking the limit therein, we obtain $\theta^{\ast} \le c^{\top}x^{\ast}$.
Therefore, we have $c^{\top} x^{\ast} = \theta^{\ast}$.
Thus, we conclude that $x^{\ast}$ is optimal for the SOCP.
\qed
\end{proof}

\begin{remark}
According to \cite[Theorem~3.1]{KMS13},
for the case with $\calK = \bbR_{+}^{n}$,
the number of iterations of the algorithm is, at most, the number of faces of the cone $\bar{A} \calK$.
\end{remark}

\section{Dual algorithm}
\label{sect: dual}

\subsection{Description of the algorithm}
In Section~\ref{sect: conv anal}, we proposed the ALPN method for solving SOCP~\eqref{eq:SOCPstd}.
In this section, we consider a dual algorithm for the ALPN method,
which solves the dual problem of SOCP~\eqref{eq:SOCPstd} in dual variables $y\in\bbR^m$:
\begin{equation}
\begin{array}{ll}
\text{minimize}&	b^{\top} y\\
\text{subject to}&	A^{\top} y - c \in \calK.
\end{array}
\label{eq:SOCPd}
\end{equation}
In the dual algorithm,
the following property plays a crucial role.

\begin{proposition}
Let $x^{\ast}$ be an optimum of SOCP~\eqref{eq:SOCPstd} and 
$H^{\ast}$ be a supporting hyperplane of $\bar{A} \calK$ at $\bar{A} x^{\ast} \in \bar{A} \calK$.
Suppose that $(1, -({y}^{\ast})^{\top})^{\top}$ is a normal vector to $H^{\ast}$.
Then, together with $x^{\ast}$, ${y}^{\ast}$ and $\eta := A^{\top}{y}^{\ast} - c$ satisfy the Karush--Kuhn--Tucker (KKT) conditions
of SOCP~\eqref{eq:SOCPstd}:
%%\begin{equation}
%%b + A \bar{y} = 0,\ \eta \in \calK,\ A^{\top} y - c \in \calK,\ \eta^{\top} (A^{\top} \lambda - c) = 0
%%\end{equation}
\begin{equation}
A x^{\ast} = b,\
-c + A^{\top} y^{\ast} - \eta = 0,\
\eta \in \calK,\
x^{\ast} \in \calK,\
\eta^{\top} x^{\ast} = 0.
%b + A \eta = 0,\ \eta \in \calK,\ A^{\top} y - c \in \calK,\ \eta^{\top} (A^{\top} \lambda - c) = 0
\label{dlsec:KKT}
\end{equation}
In particular, $y^{\ast}$ is an optimum of the dual SOCP~\eqref{eq:SOCPd}.
\end{proposition}

\begin{proof}
As $H^{\ast}$ is a supporting hyperplane of
$\bar{A} \calK$ at $\bar{A} x^{\ast} = (c^{\top} x^{\ast}, (A x^{\ast})^{\top})^{\top}$
and $x^{\ast}$ solves SOCP~\eqref{eq:SOCPstd}, we have
\begin{equation}
\bar{A} x^{\ast} \in \argmax_{w \in \bar{A} \calK}\
(1, -(y^{\ast})^{\top}) w.
\notag
\end{equation}
Hence, it holds that
\begin{equation}
x^{\ast} \in \argmax_{x \in \calK} (1, -(y^{\ast})^{\top})
\begin{pmatrix}
c^{\top}x\\
Ax
\end{pmatrix}.
\label{dlsec:prop1}
\end{equation}
By the KKT conditions from \eqref{dlsec:prop1},
there exists some $\eta \in \bbR^n$ such that
\begin{equation}
-c + A^{\top} y^{\ast} - \eta = 0,\
\eta \in \calK,\
x^{\ast} \in \calK,\
\eta^{\top} x^{\ast} = 0,
\end{equation}
which, together with $A x^{\ast} = b$, implies Equation~\eqref{dlsec:KKT}.
The optimality of $y^{\ast}$ for SOCP~\eqref{eq:SOCPd} is obvious.
\qed
\end{proof}

Our dual algorithm is described in Algorithm~\ref{alg:DALPN},
where $\{x^{(k)}\}$ and $\{\gamma^{(k)}\}$ represent sequences generated by the ALPN method.

\begin{algorithm}
\caption{Dual adaptive LP-Newton method for SOCP~\eqref{eq:SOCPstd}}
\label{alg:DALPN}
\begin{algorithmic}[1]
\State{Set $k := 0$}
\Loop
	\If{$\gamma^{(k)}-c^{\top}x^{(k)}\neq 0$}
		\State{Set
		\begin{equation}
		y^{(k)} = -\frac{b - A x^{(k)}}{\gamma^{(k)} - c^{\top} x^{(k)}}
		\end{equation}
		}
	\ElsIf{$A x^{(k)} \neq b$}
		\State{\textbf{break}}
		\Comment{The dual problem is unbounded.}
	\Else
		\State{Set $y^{(k)} = 0 \in \bbR^{m}$.}
	\EndIf
	\State{Set $\eta^{(k)} = A^{\top} y^{(k)} - c$.}
	\If{$(x^{(k)}, y^{(k)}, \eta^{(k)})$ satisfies the KKT conditions~\eqref{dlsec:KKT}}
		\State{\Return{} $y^{k}$}
		\Comment{$y^{k}$ is a dual optimum of SOCP~\eqref{eq:SOCPstd}.}
	\EndIf
	\State{$k := k + 1$.}
\EndLoop
\end{algorithmic}
\end{algorithm}

\subsection{Convergence analysis}
In addition to Assumption~\ref{H1}, 
we make the following assumption:

\begin{assumption}
\label{H2}
Slater's constraint qualification holds for SOCP~\eqref{eq:SOCPstd}, i.e., there exists some $z\in \bbR^{n}$ such that
$z \in \interior \calK$ and $A z = b$, and the matrix $A$ is of full row rank.
\end{assumption}

Under Assumptions~\ref{H1} and \ref{H2}, it is guaranteed that
the optimal set of SOCP~\eqref{eq:SOCPd} is nonempty and compact.

\begin{theorem}
Under Assumptions~\ref{H1} and \ref{H2},
the generated sequence $\{y^{(k)}\}$ is bounded and
any accumulation point of $\{y^{(k)}\}$ solves SOCP~\eqref{eq:SOCPd}.
\end{theorem}

\begin{proof}
We first show the former claim.
To construct a contradiction, suppose that $\{y^{(k)}\}$ is unbounded, and hence 
there exists some subsequence $\{y^{(k)}\}_{k \in S}$ such that $\|y^{(k)}\| \to \infty$ and $y^{(k)} \neq 0$ for any $k \in S$.
Note that $H^{(k)}$ is a supporting hyperplane of $\bar{A} \calK_{E^{(k)}}$ at $\bar{A} x^{(k)}$
that has the normal vector $(1, -(y^{(k)})^{\top})^{\top}$.
Then, by the construction of $x^{(k)}$, we find that
\begin{equation}
\bar{A} x^{(k)} \in \argmax_{w \in \bar{A} \calK_{E^{(k)}}} (1, -(y^{(k)})^{\top})^{\top} w,
\end{equation}
and thus
\begin{equation}
x^{(k)} \in \argmax_{x \in \calK_{E^{(k)}}} c^{\top} x - (y^{(k)})^{\top} A x.
\end{equation}
Under the KKT conditions and the definition of $x^{(k)}$,
we have
\begin{equation}
-c + A^{\top} y^{(k)} \in \calK_{E^{(k)}}^{\ast},\
x^{(k)} \in \calK_{E^{(k)}},\
(-c + A^{\top} y^{(k)})^{\top} x^{(k)} = 0,\
A x^{(k)} = b^{(k)},
\label{secdual:thm:KKT}
\end{equation}
where $\calK_{E^{(k)}}^{\ast}$ denotes the dual cone of $\calK_{E^{(k)}}$.
As $\calK_{E^{(k)}}^{\ast} \subseteq \calK^{\ast} = \calK$ follows from $\calK \subseteq \calK_{E^{(k)}}$,
we find that $-c + A^{\top} y^{(k)} \in \calK$.
Divide $-c + A^{\top} y^{(k)} \in \calK$ and
$(-c + A^{\top} y^{(k)})^{\top} x^{(k)} = 0$ by $\|y^{(k)}\|$ and let $k \in S \to \infty$ in Equation~\eqref{secdual:thm:KKT}.
Choose an accumulation point of $\{y^{(k)}/\|y^{(k)}\|\}$ and denote it by $d^{\ast}$.
Let $x^{\ast}$ be an accumulation point of $\{x^{(k)}\}$ (recall Theorem~\ref{thm:xopt}).
Without loss of generality, we may assume that $\lim_{k \in S \to \infty} (x^{(k)}, y^{(k)} / \|y^{(k)}\|) = (x^{\ast}, d^{\ast})$.
Then, noting that $(-c + A^{\top} y^{(k)}) / \|y^{(k)}\| \in \calK$ for any $k \in S$ and $\lim_{k \to \infty} b^{(k)} = b$ by Lemma~\ref{lem:limb},
it holds that
\begin{equation}
A^{\top} d^{\ast} \in \calK,\
x^{\ast} \in \calK,\
(A x^{\ast})^{\top} d^{\ast} = 0,\
A x^{\ast} = b.
\label{secdual:thm:eq1}
\end{equation}
Here, let $\bar{y}$ be an arbitrary optimum of SOCP~\eqref{eq:SOCPd}.
Then, the set $\Omega := \{\bar{y} + s d^{\ast}: s \ge 0\}$ is contained by the optimal solution set of SOCP~\eqref{eq:SOCPd},
denoted by $\mathcal{S}_{\mathrm{opt}}^{\mathrm{D}}$.
This is shown as follows.
Fix an arbitrary value of $s \ge 0$.
Using the first relation in \eqref{secdual:thm:eq1} and $-c + A^{\top} \bar{y} \in \calK$,
we have $-c + A^{\top} (\bar{y} + s d^{\ast}) \in \calK$,
and so $\bar{y} + s d^{\ast}$ is feasible for SOCP~\eqref{eq:SOCPd}.
Moreover, it can be deduced from \eqref{secdual:thm:eq1} that
\begin{equation}
b^{\top} (\bar{y} + s d^{\ast})
= b^{\top} \bar{y} + s (A x^{\ast})^{\top} d^{\ast}
= b^{\top} \bar{y},
\end{equation}
which indicates that the optimal value of SOCP~\eqref{eq:SOCPd} is also attained at $\bar{y} + s d^{\ast}$.
Therefore, $\Omega \subseteq \mathcal{S}_{\mathrm{opt}}^{\mathrm{D}}$.
Note that $\Omega$ is unbounded because $\|d^{\ast}\| \neq 0$,
which implies the unboundedness of $\mathcal{S}_{\mathrm{opt}}^{\mathrm{D}}$.
However, this contradicts the boundedness of $\mathcal{S}_{\mathrm{opt}}^{\mathrm{D}}$ derived from Assumptions~\ref{H1} and \ref{H2}.
Hence, $\{y^{(k)}\}$ is bounded.

The second part of the claim is easy to prove by taking the limit in Equation~\eqref{secdual:thm:KKT} with the first relation
replaced by $-c + A^{\top} y^{(k)} \in \calK$.
%%Let arbitrarily chosen accumulation point of $\{y^{(k)}\}$ be $y^{\ast}$.
%%Without loss of generality, we may assume that $y^{(k)}\to y^{\ast}$ as $k\to\infty$.
%%Letting $k$ goes to $\infty$ in \eqref{secdual:thm:KKT}, 
\qed
\end{proof}

\section{Numerical Results}
\label{sect: expr}

%\subsection{Ill-conditioned Instance}
%We solved the following ill-conditioned SOCP~\cite{Fri16}:
%\begin{equation}
%\begin{array}{ll}
%\text{minimize}&	2 x_{3}\\
%\text{subject to}&	x_{1} + x_{2} \le 0,\\
%&					x_{3} \ge -1,\\
%&					(x_{1}, x_{2}, x_{3}) \in \calK^{3}.
%\end{array}
%\label{eq: Fri16P}
%\end{equation}
%Dual problem of Problem~\eqref{eq: Fri16P} is
%\begin{equation}
%\begin{array}{ll}
%\text{minimize}&	y_{2}\\
%\text{subject to}&	y_{1} = -s_{1},\\
%&					y_{1} = -s_{2},\\
%&					y_{2} - 2 = -s_{3},\\
%&					y_{1} \le 0,\\
%&					y_{2} \ge 0,\\
%&					(s_{1}, s_{2}, s_{3}) \in \calK^{3}.
%\end{array}
%\label{eq: Fri16D}
%\end{equation}
%The optimal values of the primal and the dual is 0 and 2, respectively.

We conducted numerical experiments to verify the performance of our proposed algorithm.
We implemented the ALPN method with MATLAB R2018a (9.4.0.813654)
on a workstation running CentOS release 6.10 with eight Intel Xeon CPUs (E3-1276 v3 3.60~GHz) and 32~GB RAM.

We used an initial polyhedral approximation of the $i$-th block of $\calK$ with $n_{i} \ge 2$ given by
\begin{equation}
E_{i}^{(0)} := \{\pm e_{j} \in \bbR^{n_{i} - 1}: j = 1, 2, \dots, n_{i} - 1\},
\end{equation}
where $e_{j}$ denotes the $j$-th column of the identity matrix.
Note that $E_{i}^{(0)}$ defined above exactly represents $\calK_{i}$ if $n_{i} = 1, 2$.
In the projection step, we solved Problem~\eqref{eq: projection} using the MATLAB function \texttt{lsqlin}.
We stopped the algorithm when an approximate primal optimal solution was found,
namely, a primal solution~$x^{(k)}$ at the $k$-th iteration satisfies
\begin{equation}
\max\{\|A x^{(k)} - b\|, \max_{i = 1}^{p} \|\bar{x}^{i, (k)}\| - x_{1}^{i, (k)}\} \le 10^{-4}.
\end{equation}

We randomly generated the following instances of SOCP~\eqref{eq:SOCPstd}.
First, we set $m$, $p$, and $(n_{1}, n_{2}, \dots, n_{p})$
and randomly generated each element of $A$ from the standard Gaussian distribution.
Next, we set $b = A \tilde{x}$ and $c = A^{\top} e - \tilde{s}$,
where $e$ is the vector whose elements are all ones and
\begin{equation}
\tilde{x}^{i} = \tilde{s}^{i} := e_{1} \in \bbR^{n_{i}}\ (i = 1, 2, \dots, p).
\end{equation}
Note that the two points~$\tilde{x}$ and $\tilde{s}$
%defined above 
are interior feasible solutions of the primal and dual problems, respectively.

\subsection{Performance of the Adaptive LP-Newton method}

\begin{table}
\scriptsize
\caption{Performance of the Adaptive LP-Newton method.}
\label{tbl: expr}
\centering
\begin{tabular}{cccrrrr}
\toprule
\multicolumn{3}{c}{\# dimensions}&	&	&	\multicolumn{2}{c}{\# hyperplanes}\\
\cmidrule(rl){1-3}
\cmidrule(rl){6-7}
$m$&	$n$&	$(n_{1}, n_{2}, \dots, n_{p})$&
\multicolumn{1}{c}{time [s]}&	\multicolumn{1}{c}{\# iter}&
\multicolumn{1}{c}{initial}&	\multicolumn{1}{c}{final}\\
\midrule
$10$&	$200$&	$(1,1,\dots,1)$&	    0.1&	    3.5&	  200.0&	  200.0\\%	10
$10$&	$200$&	$(2,2,\dots,2)$&	    0.3&	    3.8&	  200.0&	  385.5\\%	10
%$10$&	$200$&	$(3,\dots,3,1,1)$&	    1.1&	    7.2&	  266.0&	  623.6\\%	10
$10$&	$200$&	$(5,5,\dots,5)$&	    7.3&	   19.1&	  320.0&	 1019.4\\%	10
$10$&	$200$&	$(10,10,\dots,10)$&	   58.4&	   62.5&	  360.0&	 1580.9\\%	10
$10$&	$200$&	$(20,20,\dots,20)$&	  167.6&	  141.1&	  380.0&	 1779.4\\%	10
$10$&	$200$&	$(100,100)$&	  136.8&	  352.0&	  396.0&	 1098.0\\%	10
$10$&	$200$&	$200$&	   48.1&	  274.4&	  398.0&	  671.4\\%	10
\midrule
$10$&	$350$&	$(1,1,\dots,1)$&	    0.1&	    3.6&	  350.0&	  350.0\\%	10
$10$&	$350$&	$(2,2,\dots,2)$&	    0.9&	    3.5&	  350.0&	  592.8\\%	10
%$10$&	$350$&	$(3,\dots,3,1,1)$&	    4.7&	    7.1&	  466.0&	 1072.3\\%	10
$10$&	$350$&	$(5,5,\dots,5)$&	   34.0&	   19.2&	  560.0&	 1786.6\\%	10
$10$&	$350$&	$(10,10,\dots,10)$&	  303.0&	   64.7&	  630.0&	 2849.2\\%	10
$10$&	$350$&	$(35,35,\dots,35)$&	 1911.0&	  281.9&	  680.0&	 3489.0\\%	10
$10$&	$350$&	$(175,175)$&	  462.4&	  414.7&	  696.0&	 1523.4\\%	10
$10$&	$350$&	$350$&	  104.4&	  217.3&	  698.0&	  914.3\\%	10
\midrule
$10$&	$500$&	$(1,1,\dots,1)$&	    0.1&	    3.8&	  500.0&	  500.0\\%	10
$10$&	$500$&	$(2,2,\dots,2)$&	    2.0&	    3.5&	  500.0&	  805.7\\%	10
%$10$&	$500$&	$(3,\dots,3,1,1)$&	   12.8&	    7.4&	  666.0&	 1598.1\\%	10
$10$&	$500$&	$(5,5,\dots,5)$&	   98.0&	   19.9&	  800.0&	 2638.5\\%	10
$10$&	$500$&	$(10,10,\dots,10)$&	  856.6&	   66.0&	  900.0&	 4135.5\\%	10
$10$&	$500$&	$(50,50,\dots,50)$&	 5618.0&	  372.5&	  980.0&	 4694.8\\%	10
$10$&	$500$&	$(250,250)$&	  866.0&	  398.1&	  996.0&	 1790.1\\%	10
$10$&	$500$&	$500$&	  183.6&	  172.7&	  998.0&	 1169.7\\%	10
\midrule
$50$&	$200$&	$(1,1,\dots,1)$&	    0.0&	    4.0&	  200.0&	  200.0\\%	10
$50$&	$200$&	$(2,2,\dots,2)$&	    0.3&	    4.0&	  200.0&	  461.8\\%	10
%$50$&	$200$&	$(3,\dots,3,1,1)$&	    1.2&	    7.4&	  266.0&	  672.4\\%	10
$50$&	$200$&	$(5,5,\dots,5)$&	    5.2&	   16.8&	  320.0&	  947.6\\%	10
$50$&	$200$&	$(10,10,\dots,10)$&	   26.3&	   47.9&	  360.0&	 1297.9\\%	10
$50$&	$200$&	$(20,20,\dots,20)$&	   67.3&	  105.3&	  380.0&	 1423.0\\%	10
$50$&	$200$&	$(100,100)$&	   62.7&	  257.1&	  396.0&	  908.2\\%	10
$50$&	$200$&	$200$&	   30.5&	  238.1&	  398.0&	  635.1\\%	10
\midrule
$50$&	$350$&	$(1,1,\dots,1)$&	    0.1&	    4.0&	  350.0&	  350.0\\%	10
$50$&	$350$&	$(2,2,\dots,2)$&	    1.3&	    3.9&	  350.0&	  752.3\\%	10
%$50$&	$350$&	$(3,\dots,3,1,1)$&	    5.6&	    7.4&	  466.0&	 1180.9\\%	10
$50$&	$350$&	$(5,5,\dots,5)$&	   27.3&	   17.2&	  560.0&	 1690.5\\%	10
$50$&	$350$&	$(10,10,\dots,10)$&	  184.6&	   54.4&	  630.0&	 2499.0\\%	10
$50$&	$350$&	$(35,35,\dots,35)$&	  675.8&	  200.5&	  680.0&	 2674.9\\%	10
$50$&	$350$&	$(175,175)$&	  235.7&	  308.3&	  696.0&	 1310.6\\%	10
$50$&	$350$&	$350$&	   64.0&	  164.6&	  698.0&	  861.6\\%	10
\midrule
$50$&	$500$&	$(1,1,\dots,1)$&	    0.1&	    3.9&	  500.0&	  500.0\\%	10
$50$&	$500$&	$(2,2,\dots,2)$&	    3.2&	    3.8&	  500.0&	 1026.4\\%	10
%$50$&	$500$&	$(3,\dots,3,1,1)$&	   15.1&	    7.4&	  666.0&	 1681.7\\%	10
$50$&	$500$&	$(5,5,\dots,5)$&	   80.1&	   17.9&	  800.0&	 2487.1\\%	10
$50$&	$500$&	$(10,10,\dots,10)$&	  548.5&	   56.4&	  900.0&	 3669.7\\%	10
$50$&	$500$&	$(50,50,\dots,50)$&	 2287.4&	  274.9&	  980.0&	 3719.0\\%	10
$50$&	$500$&	$(250,250)$&	  392.6&	  281.0&	  996.0&	 1556.0\\%	10
$50$&	$500$&	$500$&	  125.5&	  140.3&	  998.0&	 1137.3\\%	10
\midrule
$100$&	$200$&	$(1,1,\dots,1)$&	    0.0&	    4.0&	  200.0&	  200.0\\%	10
$100$&	$200$&	$(2,2,\dots,2)$&	    0.3&	    4.0&	  200.0&	  481.9\\%	10
%$100$&	$200$&	$(3,\dots,3,1,1)$&	    1.1&	    7.0&	  266.0&	  658.1\\%	10
$100$&	$200$&	$(5,5,\dots,5)$&	    3.8&	   14.9&	  320.0&	  875.7\\%	10
$100$&	$200$&	$(10,10,\dots,10)$&	   13.6&	   38.1&	  360.0&	 1102.0\\%	10
$100$&	$200$&	$(20,20,\dots,20)$&	   30.3&	   78.4&	  380.0&	 1154.0\\%	10
$100$&	$200$&	$(100,100)$&	   35.3&	  195.4&	  396.0&	  784.8\\%	10
$100$&	$200$&	$200$&	   24.5&	  211.5&	  398.0&	  608.5\\%	10
\midrule
$100$&	$350$&	$(1,1,\dots,1)$&	    0.1&	    4.0&	  350.0&	  350.0\\%	10
$100$&	$350$&	$(2,2,\dots,2)$&	    1.5&	    4.0&	  350.0&	  828.4\\%	10
%$100$&	$350$&	$(3,\dots,3,1,1)$&	    5.6&	    7.2&	  466.0&	 1174.6\\%	10
$100$&	$350$&	$(5,5,\dots,5)$&	   24.1&	   16.6&	  560.0&	 1645.8\\%	10
$100$&	$350$&	$(10,10,\dots,10)$&	  114.6&	   46.9&	  630.0&	 2236.5\\%	10
$100$&	$350$&	$(35,35,\dots,35)$&	  397.4&	  166.1&	  680.0&	 2331.0\\%	10
$100$&	$350$&	$(175,175)$&	  159.0&	  257.8&	  696.0&	 1209.6\\%	10
$100$&	$350$&	$350$&	   87.9&	  212.3&	  698.0&	  909.3\\%	10
\midrule
$100$&	$500$&	$(1,1,\dots,1)$&	    0.1&	    4.0&	  500.0&	  500.0\\%	10
$100$&	$500$&	$(2,2,\dots,2)$&	    5.9&	    4.0&	  500.0&	 1195.9\\%	10
%$100$&	$500$&	$(3,\dots,3,1,1)$&	   16.8&	    7.2&	  666.0&	 1683.3\\%	10
$100$&	$500$&	$(5,5,\dots,5)$&	   72.2&	   17.0&	  800.0&	 2399.6\\%	10
$100$&	$500$&	$(10,10,\dots,10)$&	  388.0&	   50.2&	  900.0&	 3360.0\\%	10
$100$&	$500$&	$(50,50,\dots,50)$&	 1286.0&	  223.6&	  980.0&	 3206.0\\%	10
$100$&	$500$&	$(250,250)$&	  399.2&	  290.1&	  996.0&	 1574.2\\%	10
$100$&	$500$&	$500$&	  120.4&	  139.4&	  998.0&	 1136.4\\%	10
\bottomrule
\end{tabular}
\end{table}

For the ALPN method, Table~\ref{tbl: expr} presents the average runtime,
number of iterations, and number of hyperplanes in the initial and final approximations of the SOC over ten executions.
From this table, we can make the following observations:
%observe the followings:
\begin{itemize}
	\item When $\calK$ is \textit{polyhedral-like,} i.e., $p \approx n$ and $n_{i} \approx 1$ for all $i = 1, 2, \dots, p$,
	the ALPN method works well.
	The algorithm gives a good polyhedral approximation of $\calK$ with a small number of hyperplanes.
	As a result, there are few iterations and the computation time is short.
	\item When $\calK$ is \textit{medium-dimensional,} i.e., $1 \ll p \ll n$ and $1 \ll n_{i} \ll n$ for all $i = 1, 2, \dots, p$,
	the ALPN method becomes slow,
	although it gets better
	when $\calK$ is \textit{high-dimensional,} i.e., $p \approx 1$ and $n_{i} \approx n$ for all $i = 1, 2, \dots, p$.
	The medium-dimensional $\calK$ requires many hyperplanes to obtain a good polyhedral approximation.
	\item The total dimension~$n$ of the variables seems to be positively correlated with the runtime, 
	although the runtime of the original LP-Newton method for LP is almost independent of $n$~\cite{FHYZ09}.
	This difference arises from the solution methods of the minimum norm point step.
	In our implementation, we solve Problem~\eqref{eq: projection} using the MATLAB function \texttt{lsqlin},
	for which the computation time depends on $n$.
	\item Surprisingly, the number~$m$ of linear constraints is \textit{negatively} correlated with the runtime.
	This might be because the dimension of the feasible region is low for large values of $m$,
	and this region can then be approximated by a small number of hyperplanes.
\end{itemize}

\subsection{Comparison with the primal-dual interior-point method}

\begin{table}
\caption{Comparison with the primal-dual interior-point method.}
\label{tbl: comparison}
\centering
\begin{tabular}{cccrr}
\toprule
\multicolumn{3}{c}{\# dimensions}&	\multicolumn{2}{c}{time [s]}\\
\cmidrule(rl){1-3}
\cmidrule(rl){4-5}
$m$&	$n$&	$(n_{1}, n_{2}, \dots, n_{p})$&
\multicolumn{1}{c}{ALPN}&	\multicolumn{1}{c}{SDPT3}\\
\midrule
$1400$&	$1500$&	$(3,3,\dots,3)$&	  177.3&	  366.6\\
$1700$&	$1800$&	$(3,3,\dots,3)$&	  260.4&	  638.4\\
$2000$&	$2100$&	$(3,3,\dots,3)$&	  363.4&	  970.0\\
\bottomrule
\end{tabular}
\end{table}

We also compared our proposed ALPN method with the primal-dual interior-point method.
In this experiment, we solved the randomly generated instances using
our implementation of ALPN and SDPT3~\cite{TTT03},
which is a MATLAB implementation of the primal-dual interior-point method.
Basically, SDPT3 was found to be faster than the ALPN method.
However, the computation time of SDPT3 increases with $m$, whereas that of ALPN decreases as $m$ and $p$ increase.
For instances with large values of $m$ and $p$, ALPN outperformed SDPT3.
The results are presented in Table~\ref{tbl: comparison},
which shows the average runtime of the ALPN and SDPT3 methods over ten runs.

\section{Conclusions}
In this paper, we have developed an LP-Newton method for SOCP
through a transformation into LSIP with an infinite number of linear inequality constraints.
The proposed ALPN algorithm
produces a sequence by sequentially projecting the current point onto a polyhedral cone
%, called zonotope,
arising from finitely many linear inequality constraints chosen from the constraints of the LSIP.
We also proposed a dual algorithm for the ALPN method for solving the dual of the SOCP.
Under some mild assumptions, 
we proved that arbitrary accumulation points of the sequences generated by the two proposed algorithms are optima of the SOCP and its dual.
Finally, we conducted some numerical experiments and compared the performance of our algorithms with that of the primal-dual interior point method.
%There still remains some issues which should be settled in the near future.
%For example, 
Future work will consider the extension of the ALPN method to semi-definite programming problems or symmetric cone programming problems.

%\begin{acknowledgements}
%We would like to thank to Dr.\ Bruno Figueira Louren\c{c}o for useful comments.
%\end{acknowledgements}

\bibliographystyle{spmpsci}      % mathematics and physical sciences
\bibliography{adaptiveLPNewton}   % name your BibTeX data base

\begin{thebibliography}{10}
\providecommand{\url}[1]{{#1}}
\providecommand{\urlprefix}{URL }
\expandafter\ifx\csname urlstyle\endcsname\relax
  \providecommand{\doi}[1]{DOI~\discretionary{}{}{}#1}\else
  \providecommand{\doi}{DOI~\discretionary{}{}{}\begingroup
  \urlstyle{rm}\Url}\fi

\bibitem{faraut1994analysis}
Faraut, J., Koranyi, A.: Analysis on Symmetric Cones.
\newblock Oxford University Press (1994)

\bibitem{FHYZ09}
Fujishige, S., Hayashi, T., Yamashita, K., Zimmermann, U.: Zonotopes and the
  {LP}-{N}ewton method.
\newblock Optimization and Engineering \textbf{10}, 193--205 (2009)

\bibitem{hayashi2016simplex}
Hayashi, S., Okuno, T., Ito, Y.: Simplex-type algorithm for second-order cone
  programmes via semi-infinite programming reformulation.
\newblock Optimization Methods and Software \textbf{31}(6), 1272--1297 (2016)

\bibitem{hayashi1}
Hayashi, S., Yamashita, N., Fukushima, M.: A combined smoothing and
  regularization method for monotone second-order cone complementarity
  problems.
\newblock SIAM Journal on Optimization \textbf{15}, 593--615 (2005)

\bibitem{sip2}
Hettich, R., Kortanek, K.O.: {Semi-Infinite Programming: Theory, Methods, and
  Applications}.
\newblock SIAM Review \textbf{35}, 380--429 (1993)

\bibitem{KMS13}
Kitahara, T., Mizuno, S., Shi, J.: The {LP}-{N}ewton method for standard form
  linear programming problems.
\newblock Operations Research Letters \textbf{41}, 426--429 (2013)

\bibitem{kitahara2018extension}
Kitahara, T., Tsuchiya, T.: An extension of {Chubanov}'s polynomial-time linear
  programming algorithm to second-order cone programming.
\newblock Optimization Methods and Software \textbf{33}(1), 1--25 (2018)

\bibitem{lobo1998applications}
Lobo, M.S., Vandenberghe, L., Boyd, S., Lebret, H.: Applications of
  second-order cone programming.
\newblock Linear algebra and its applications \textbf{284}(1-3), 193--228
  (1998)

\bibitem{sip1}
L{\'o}pez, M., Still, G.: Semi-infinite programming.
\newblock European Journal of Operational Research \textbf{180}, 491--518
  (2007)

\bibitem{monteiro2000polynomial}
Monteiro, R.D., Tsuchiya, T.: Polynomial convergence of primal-dual algorithms
  for the second-order cone program based on the {MZ}-family of directions.
\newblock Mathematical programming \textbf{88}(1), 61--83 (2000)

\bibitem{muramatsu2006pivoting}
Muramatsu, M.: A pivoting procedure for a class of second-order cone
  programming.
\newblock Optimization Methods and Software \textbf{21}(2), 295--315 (2006)

\bibitem{SR17}
Silvestri, F., Reinelt, G.: The {LP}-{N}ewton method and conic optimization.
\newblock arXiv:1611.09260v2 (2017)

\bibitem{TTT03}
T\"{u}t\"{u}nc\"{u}, R.H., Toh, K.C., Todd, M.J.: Solving
  semidefinite-quadratic-linear programs using {SDPT3}.
\newblock Mathematical Programming \textbf{95}, 189--217 (2003)

\bibitem{Wil76}
Wilhelmsen, D.R.: A nearest point algorithm for convex polyhedral cones and
  applications to positive linear approximation.
\newblock Mathematics of Computation \textbf{30}, 48--57 (1976)

\bibitem{Wol76}
Wolfe, P.: Finding the nearest point in a polytope.
\newblock Mathematical Programming \textbf{11}, 128--149 (1976)

\end{thebibliography}
\end{document}